\documentclass[a4paper,12pt]{article}
\usepackage[english]{babel}
\usepackage{amsmath,amsthm,amsfonts,amssymb}
\usepackage{indentfirst}
\usepackage[numbers]{natbib}
\usepackage{a4wide}
\usepackage{url,icomma}
\usepackage{graphicx}
\usepackage{array}
\usepackage{booktabs}
\usepackage[small,bf]{caption}

\theoremstyle{plain}
\newtheorem{teo}{Theorem}[section]

\newtheorem{lem}[teo]{Lemma}

\theoremstyle{definition}

\newtheorem*{obs}{Remark}

\numberwithin{equation}{section}
\newcolumntype{C}{>{\centering\arraybackslash$}c<{$}}

\newcommand{\bbR}{\mathbb{R}}

\newcommand{\Xfim}{X^{(N)}_{F}}
\newcommand{\Vfim}{V^{(N)}_{F}}
\newcommand{\LW}[1]{\ensuremath{W_0(#1)}}
\newcommand{\chao}[1]{\lfloor{#1}\rfloor}
\newcommand{\teto}[1]{\lceil{#1}\rceil}
\newcommand{\parte}[1]{\textnormal{({#1})}}
\newcommand{\xf}{x_{\infty}}
\newcommand{\vf}{v_{\infty}}
\newcommand{\rf}{\rho}
\newcommand{\je}{j^{\ast}}
\newcommand{\chnx}{\chao{N x}}
\newcommand{\tenx}{\teto{N x}}

\newcommand{\tenu}{\teto{N u}}
\newcommand{\chnv}{\chao{N v}}
\newcommand{\tenv}{\teto{N v}}
\newcommand{\chnxf}{\chao{N \xf}}
\newcommand{\tenxf}{\teto{N \xf}}

\newcommand{\tenvf}{\teto{N \vf}}
\newcommand{\lpx}[1][x]{h({#1})}
\newcommand{\ldx}[1][x]{H({#1})}
\newcommand{\lpv}[1][v]{\psi({#1})}
\newcommand{\med}{\nu_N}
\newcommand{\ic}[1][K]{I({#1})}
\newcommand{\Gzu}{G \cap [0, 1)}
\newcommand{\Gui}{G \cap [1, \infty)}
\newcommand{\Gzx}{G \cap [0, \xf)}
\newcommand{\Gxv}{G \cap (\xf, \vf]}
\newcommand{\Gvu}{G \cap (\vf, 1)}
\newcommand{\cam}{\gamma}
\newcommand{\ccam}{\Gamma}
\newcommand{\fn}[1][j]{\phi^{(N)}({#1})}
\newcommand{\dn}[1][j]{\Delta^{(N)}({#1})}
\newcommand{\xn}[1][j]{\Xi^{(N)}({#1})}
\newcommand{\fmn}[1][z]{\Gamma^{(N)}({#1})}
\newcommand{\lan}[1][z]{\Lambda^{(N)}({#1})}
\newcommand{\lal}[1][z]{\Lambda({#1})}
\newcommand{\lopp}{\lambda^{(N)}(x)}
\newcommand{\elopp}{\tilde{\lambda}^{(N)}(x)}
\newcommand\stir[2]{\genfrac{\{}{\}}{0pt}{}{#1}{#2}}
\newcommand{\ogr}[1]{\mathcal{O}\left({#1}\right)}
\newcommand{\aut}{\mathcal{A}_n}
\newcommand{\tstru}{\mathcal{D}_n}
\newcommand{\al}{\alpha}
\newcommand{\be}{\beta}
\newcommand{\Ek}{\kappa}
\newcommand{\citparen}[1]{\citet{#1}}
\newcommand{\dpref}{ }
\newcommand{\vvirg}{, }

\begin{document}
\title{A large deviations principle for the \\
Maki--Thompson rumour model}
\author{Elcio Lebensztayn\thanks{The author is thankful to CNPq (303872/2012-8), FAPESP (2012/22673-4), and PRP-FAEPEX-UNICAMP (016/2013) for financial support.}}
\date{}

\maketitle

\begin{center}
\begin{small}
Institute of Mathematics, Statistics and Scientific Computation, \\
University of Campinas -- UNICAMP, \\
Rua S\'ergio Buarque de Holanda 651, CEP 13083-859, Campinas, SP, Brazil. \\
E-mail: \texttt{lebensztayn@ime.unicamp.br} \\[1cm]
\end{small}
\end{center}

\begin{abstract}
We consider the stochastic model for the propagation of a rumour within a population which was formulated by \citet{MT}.
\citet{Sudbury} established that, as the population size tends to infinity, the proportion of the population never hearing the rumour converges in probability to $0.2032$.
\citet{Watson} later derived the asymptotic normality of a suitably scaled version of this proportion.
We prove a corresponding large deviations principle, with an explicit formula for the rate function. \\[0.3cm]
\textbf{Keywords:} Large deviations, Markov process, Rumour models, Limit theorems, Final size distribution. \\[0.3cm]
\textbf{2010 Mathematics Subject Classification:} Primary: 60F10, 60J27; secondary: 60K30.
\end{abstract}

\baselineskip=18pt

\section{Introduction}
\label{S: Introduction}

Rumours play an important role in various aspects of the human life: in social relationships, politics, economy, diplomacy, marketing.
Two classical models of the mathematical literature for the spread of a rumour within a population were introduced by~\citet{DK} and~\citet{MT}.
In these models, a closed homogeneously mixing population of $N + 1$ individuals is considered.
The essential assumption is that an individual aware of the rumour will go on propagating it until the first time when he meets another person who also knows the rumour.
At this moment, this individual feels that there is no longer thrill in passing on the rumour.
Thus, the population is subdivided into three classes: \textit{ignorants} (those not aware of the rumour), \textit{spreaders} (who are spreading it), and \textit{stiflers} (who know the rumour but have ceased communicating it after meeting somebody who has already heard it).
The individuals interact in a random manner, in such a way that the evolution of the population is described by a continuous-time Markov chain.
The process eventually terminates (when there are no more spreaders in the population), so it is of interest to study the proportion of remaining ignorants.
We refer the interested reader to \citet[Chapter~5]{DG} for an excellent presentation on the mathematical modelling of rumours.

In the model formulated by~\citet{MT}, which was discussed later by~\citet{Frauenthal}, the rumour is propagated by directed contact of the spreaders with other individuals in the population.
The individuals change their states over time according to a simple set of rules, described as follows.
When a spreader meets an ignorant, the rumour is told and the ignorant becomes a spreader.
If a spreader contacts another spreader or a stifler, the initiating spreader turns into a stifler. 
We adopt the usual notation, denoting respectively by $X(t)$, $Y(t)$ and $Z(t)$ the number of ignorants, spreaders and stiflers at time~$t$.
Initially, $X(0) = N$, $Y(0) = 1$ and $Z(0) = 0$, while $X(t) + Y(t) + Z(t) = N + 1$ for all~$t \geq 0$.
The process $\{(X(t), Y(t))\}_{t \geq 0}$ is a continuous-time Markov chain which proceeds according to the following transition scheme:
\begin{alignat*}{2}
&(X(t), Y(t)) \to (X(t) - 1, Y(t) + 1) 
&&\quad\text{at rate } \, X(t) \, Y(t), \\[0.1cm]
&(X(t), Y(t)) \to (X(t), Y(t) - 1) 
&&\quad\text{at rate } \, Y(t) \left(N - X(t)\right).
\end{alignat*}
The first transition corresponds to a spreader telling the rumour to an ignorant, who becomes a spreader.
The second transition corresponds to a spreader meeting another spreader or a stifler, in which case he loses the interest in propagating the rumour and becomes a stifler.

Let $\tau = \inf \{t: Y(t) = 0 \}$ be the time until the rumour process ceases, thus $X(\tau)$ is the final number of ignorants in the population.
It is convenient to write down explicitly the dependence of this quantity on~$N$, so we denote it by $\Xfim$.
The first theorems established for the Maki--Thompson model deal with the asymptotic behaviour as $N \to \infty$ of $N^{-1} \, \Xfim$ (i.e.\ the proportion of the originally ignorant individuals who remained ignorant at the end of the process).
By using a martingale technique, \citet{Sudbury} proved that
\begin{equation*}
\lim_{N\to \infty} \frac{\Xfim}{N} = \xf \approx 0.2032 \quad \text{in probability.}
\end{equation*}
Therefore, for large $N$, approximately a fifth of the people are not aware of the rumour at the moment that the spreading process stops, with high probability.
The limiting proportion of ignorants $\xf$ can be expressed in terms of the so-called Lambert $W$ function.
This is the multivalued inverse of the function $ x \mapsto x \, e^x $; see~\citet{LW} for more details.
Denoting by $W_0$ the principal branch of the Lambert $W$ function, we have that
\begin{equation}
\label{F: xi}
\xf = - \dfrac{\LW{- 2 \, e^{-2}}}{2}.
\end{equation}
\citet{Watson} later proved the corresponding Central Limit Theorem, which states that 
\begin{equation*}
\sqrt{N} \biggl(\dfrac{\Xfim}{N} - \xf\biggr) \stackrel{\mathcal{D}}{\rightarrow} \mathcal{N}(0, \sigma^2) \, \text{ as } \, N \to \infty,
\end{equation*}
where $ \stackrel{\mathcal{D}}{\rightarrow} $ denotes convergence in distribution, and $ \mathcal{N}(0, \sigma^2) $ is the Gaussian distribution with mean zero and variance given by
\begin{equation}
\label{F: VN}
\sigma^2 = \frac{\xf (1 - \xf)}{1 - 2 \, \xf} \approx 0.2727.
\end{equation}
\citet{LP} studied a variant of the Maki--Thompson model whose dynamics includes different behaviours of the individuals in front of the rumour.
By using martingales, they characterized in terms of Gontcharoff polynomials the joint distribution of the number of individuals who ultimately heard the rumour and the total personal time units during which the rumour spread.
The limit theorems proved by~\citet{Sudbury} and \citet{Watson} were generalized by \citet{RPRS} for a Maki--Thompson rumour model with general initial configuration and in which a spreader becomes a stifler only after being involved in a random number of unsuccessful telling meetings.
In~\citet{LTRM}, these limit theorems are also established for a general stochastic rumour model defined in terms of parameters that determine the rates at which the different interactions between individuals occur.
This definition allows a quantitative formulation of various behavioural mechanisms of the people involved in the dissemination of the rumour.

The main results of the paper are stated in Section~\ref{S: Main results}.
Our main theorem is that a full large deviations principle holds for the proportion $N^{-1} \, \Xfim$.
To prove this result, we first derive a closed formula for the probability mass function of $\Xfim$, revealing the structure of the distribution of this random variable in an enlightening way.
Then, we obtain the asymptotic behaviour of normalized logarithms of probabilities of certain events.
To the best of our knowledge, large deviations for the final outcome of stochastic rumour models on finite populations have never been investigated in the literature.
Regarding epidemic-like processes on infinite graphs, the main object under study is the growing set of infected individuals, and large deviations results can be found in~\citet{KRS} and references therein.

\section{Main results}
\label{S: Main results}

In our first result, we give a closed formula for the absorption probabilities of the Maki--Thompson model.
These probabilities are expressed in terms of an integer sequence~$\{ d_n \}_{n \geq 1}$ whose $n$-th term is equal to the number of certain deterministic finite automata.
An \textit{automaton} is a mathematical abstraction of a machine that performs computations on an input by moving through a series of states, and, as the result, it decides whether that input is accepted or rejected.
A complete exposition of the theory of finite automata can be found in \citet{HMU}.

For $n \geq 1$, let $d_n$ denote the number of nonisomorphic unlabelled initially connected complete and deterministic automata with $n$ states over a $2$-letter alphabet.
See \citet{BN} for more details on the formal definition of this class of automata, and \citet{DKS} for a survey about the problem of enumeration of different classes of automata.
\citet{Liskovets} proved that the sequence $\{ d_n \}_{n \geq 1}$ can be obtained recursively by
\begin{equation}
\label{F: Sequence}
\begin{aligned}
d_1 &= 1, \\[0.1cm]
d_n &= \frac{n^{2 n}}{(n - 1)!} - \sum_{i = 1}^{n - 1} \frac{n^{2 (n - i)}}{(n - i)!} \, d_i \quad \text{for } n \geq 2.
\end{aligned}
\end{equation}
Using a computer, one calculates its first elements:
\[ 1, \, 12, \, 216, \, 5248, \, 160675, \, 5931540, \, 256182290, \, \dots \]
This sequence appears as A006689 in \citet{Sloane}.
We will use in the sequel that $\{ d_n \}_{n \geq 1}$ can be expressed by the recursive formula~\eqref{F: Sequence}, as well as an asymptotic approximation and an upper bound for~$d_n$ derived from \citep{BN}.

\begin{teo}
\label{T: PAX}
For each $i = 0, \dots, N - 1$,
\[ P(\Xfim = i) = \frac{(N - 1)!}{i!} \frac{d_{N - i}}{N^{2 (N - i)}}. \]
\end{teo}

\noindent
We observe that the distribution of $\Xfim$ can be obtained from the results presented in~\citet{LP}; for more details, see the comments after the proof of Proposition~2 in this paper.
We provide a direct proof of Theorem~\ref{T: PAX}, which relies only on the analysis of the embedded chain of the process.

We need some definitions to state the full large deviations principle for the ultimate proportion of ignorants.
We define the constants
\begin{align*}
\vf &= 1 - \xf \approx 0.7968, \quad \text{and} \\[0.1cm]
\rf &= 2 + \log \xf + \log (1 - \xf) \approx 0.1792,
\end{align*}
where $\xf$ is given in~\eqref{F: xi}.
We also define the function $h: [0, 1) \rightarrow \bbR$ given by
\[ \lpx = x \log x + (1 - x) [\rf - \log (1 - x)], \]
with the usual convention that $0 \log 0 = 0$, and let $H: [0, \infty) \rightarrow [0, \infty]$ be given by
\[ \ldx =
\left\{
\begin{array}{cl}
\lpx   & \text{if } 0 \leq x < 1, \\[0.1cm]
\infty & \text{if } x \geq 1.
\end{array}	\right. \]
We note that
\begin{itemize}
\item[\parte{i}] $\lpx[\xf] = 0$.
\item[\parte{ii}] $h$ is decreasing on the intervals $[0, \xf]$ and $[\vf, 1)$, and is increasing on $[\xf, \vf]$.
\item[\parte{iii}] $h$ is strictly convex on $[0, 1/2]$, and is strictly concave on $[1/2, 1)$.
\end{itemize}
The graph of $h$ is presented in Figure~\ref{Fig: Graf h}.

\begin{figure}[!htbp]
\centering
\includegraphics[scale=0.9]{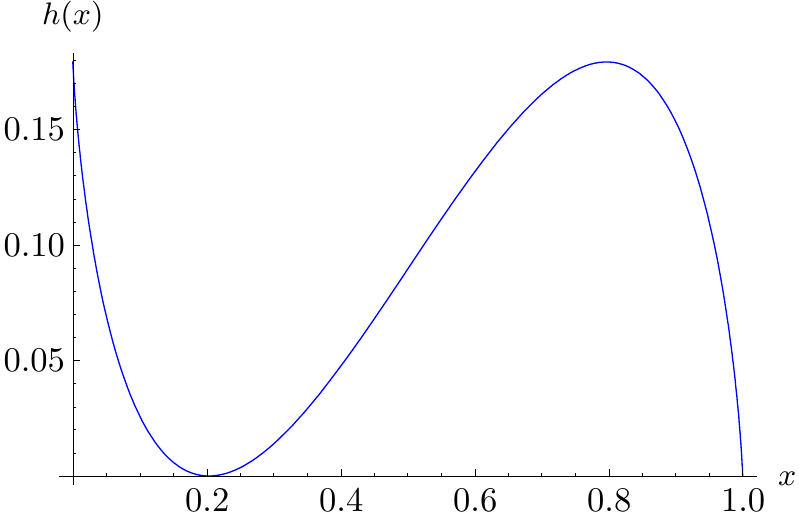}
\caption{Graph of the function $h$.}
\label{Fig: Graf h}
\end{figure}

\begin{teo}
\label{T: PGD}
Let $\med$ be the probability distribution of the random variable $N^{-1} \, \Xfim$ on~$[0, \infty)$.
Then the following conclusions hold.
\begin{itemize}
\item[\textnormal{(a)}] For each closed set $F \subset [0, \infty)$,
\begin{equation*}
\limsup_{N \to \infty} \frac{1}{N} \, \log \med(F) \leq - \inf_{x \in F} \ldx.
\end{equation*}
\item[\textnormal{(b)}] For each open set $G \subset [0, \infty)$,
\begin{equation*}
\liminf_{N \to \infty} \frac{1}{N} \, \log \med(G) \geq - \inf_{x \in G} \ldx.
\end{equation*}
\end{itemize}
\end{teo}

To finish the section, we present two results which concern the asymptotic behaviour of normalized logarithms of probabilities of certain events.
These results are useful in the proof of Theorem~\ref{T: PGD}.

\begin{teo}
\label{TA: CAX}
For every $x \in [0, 1)$, we have that
\[ \lim_{N \to \infty} -\frac{1}{N} \, \log P(\Xfim = \chnx) = 
\lim_{N \to \infty} -\frac{1}{N} \, \log P(\Xfim = \tenx) = \lpx. \]
\end{teo}

\begin{teo}
\label{TA: GDX}
\begin{itemize}
\item[\parte{a}] If $0 \leq x < \xf$, then
\begin{equation*}
\lim_{N \to \infty} -\frac{1}{N} \, \log P(\Xfim \leq N x) = \lpx.
\end{equation*}
\item[\parte{b}] If $\xf < x < y \leq \vf$, then
\begin{equation*}
\lim_{N \to \infty} -\frac{1}{N} \, \log P(N x \leq \Xfim \leq N y) = \lpx.
\end{equation*}
\item[\parte{c}] If $\vf \leq x < y < 1$, then
\begin{equation*}
\lim_{N \to \infty} -\frac{1}{N} \, \log P(N x \leq \Xfim \leq N y) = \lpx[y].
\end{equation*}
\end{itemize}
\end{teo}

\begin{obs}
Let us make a few comments about the different flavours of the limit theorems obtained for the final proportion of ignorants in Maki--Thompson model.
The Law of Large Numbers proved by \citet{Sudbury} states that, for every $x > 0$,
\begin{equation}
\label{F: LLN}
P\bigl(|\Xfim - N \xf| \geq N x\bigr) \to 0 \, \text{ as } \, N \to \infty.
\end{equation}
Therefore, the typical value of $\Xfim$ is $N \xf$.
On the other hand, the Central Limit Theorem established by \citet{Watson} asserts that, for every $x \in \bbR$,
\[ P\bigl(\Xfim - N \xf \geq x \sqrt{N}\bigr) \to 1 - \Phi(x / \sigma) \, \text{ as } \, N \to \infty, \]
where $\Phi$ is the distribution function of the standard normal distribution and $\sigma^2$ is given by~\eqref{F: VN}.
This implies that the deviations of $\Xfim$ from $N \xf$ are typically of the order~$\sqrt{N}$.
Consequently, \textit{large deviations} (of order $N$) have probabilities which tend to $0$ as $N \to \infty$ (that is, \eqref{F: LLN} holds).
The Large Deviations Theorem quantifies precisely the exponential decay rate at which this occurs, so it is a useful tool when an approximation of these small probabilities is needed.
We refer to \citet{DZ} for a detailed presentation of an historical overview and the fundamental results in this subject.

For a numerical illustration, we consider deviations from below of $\Xfim$ from the typical value.
According to Theorem~\ref{TA: GDX}, for $x \in [0, \xf)$, we have that $P(\Xfim \leq N x)$ decays exponentially in the manner of $e^{- N \, \lpx}$.
For $N \in \{1700, 1800, 1900, 2000 \}$ and $x$ taking on some values in the interval~$(0, \xf)$, we computed the normalized logarithms
\begin{equation}
\label{F: LP}
\lopp = -\frac{1}{N} \, \log P(\Xfim \leq N x),
\end{equation}
by using the exact distribution of $\Xfim$, the corresponding values $\elopp$ obtained by applying the Central Limit Theorem to estimate the probability in~\eqref{F: LP}, and the numerical values of $\lpx$.
The resulting values are shown in Table~\ref{Tab: TL}.
Notice that, for values of~$x$ that are farther from~$\xf$, better approximations for $P(\Xfim \leq N x)$ are available from the large deviations result, whereas the estimates given by the Central Limit Theorem lose accuracy.
\end{obs}

\begin{table}[htbp]
\centering
\begin{tabular}{CCCCCCCCC}
\toprule
 & \multicolumn{2}{C}{x = 0.05} & \multicolumn{2}{C}{x = 0.10} & \multicolumn{2}{C}{x = 0.15} & \multicolumn{2}{C}{x = 0.20} \\
\cmidrule[\lightrulewidth](lr){2-3}
\cmidrule[\lightrulewidth](lr){4-5}
\cmidrule[\lightrulewidth](lr){6-7}
\cmidrule[\lightrulewidth](lr){8-9}
N & \lopp & \elopp & \lopp & \elopp & \lopp & \elopp & \lopp & \elopp \\
\midrule
1700 & 0.0709 & 0.0450 & 0.0275 & 0.0213 & 0.00728 & 0.00660 & 0.000520 & 0.000538 \\
1800 & 0.0708 & 0.0449 & 0.0274 & 0.0212 & 0.00722 & 0.00654 & 0.000496 & 0.000512 \\
1900 & 0.0708 & 0.0448 & 0.0274 & 0.0211 & 0.00716 & 0.00648 & 0.000474 & 0.000489 \\
2000 & 0.0707 & 0.0448 & 0.0273 & 0.0211 & 0.00711 & 0.00643 & 0.000454 & 0.000468 \\
\midrule
\lpx & \multicolumn{2}{C}{0.0692} & \multicolumn{2}{C}{0.0259} & \multicolumn{2}{C}{0.00593} & \multicolumn{2}{C}{0.0000188} \\
\bottomrule
\end{tabular}
\caption{Values of $\lopp$, $\elopp$ and $\lpx$.}
\label{Tab: TL}
\end{table}

\section{Proofs}
\label{S: Proofs}

Theorem~\ref{T: PGD} follows from Theorem~\ref{TA: GDX}, by using standard arguments of the large deviations theory; see for instance the statement and proof of Theorem~2.2.3 in~\citparen{DZ} or Example~23.10 in~\citet{Klenke}.
We present the proof at Subsection~\ref{SS: Proof PGD}, for the sake of completeness.

To facilitate the proofs of Theorems~\ref{T: PAX}, \ref{TA: CAX} and~\ref{TA: GDX}, we restate them in terms of the random variable $\Vfim = N - \Xfim$, which represents the number of the initially ignorant individuals who heard the rumour at the end of the process.
We also define $V(t) = N - X(t)$ for $t \geq 0$, and let $\psi: (0, 1] \rightarrow \bbR$ be the function given by
\[ \lpv = (1 - v) \log (1 - v) + v [\rf - \log v]. \]
Since $\lpv = \lpx[1 - v]$ for every $v \in (0, 1]$, our task is done once we prove the following results.

\begin{teo}
\label{T: PAV}
For each $j = 1, \dots, N$,
\[ P(\Vfim = j) = \frac{(N - 1)!}{(N - j)!} \frac{d_j}{N^{2 j}}. \]
\end{teo}

\begin{teo}
\label{TA: CAV}
For every $v \in (0, 1]$, we have that
\[ \lim_{N \to \infty} -\frac{1}{N} \, \log P(\Vfim = \chnv) = 
\lim_{N \to \infty} -\frac{1}{N} \, \log P(\Vfim = \tenv) = \lpv. \]
\end{teo}

\begin{teo}
\label{TA: GDV}
\begin{itemize}
\item[\parte{a}] If $0 < u < v \leq \xf$, then
\[ \lim_{N \to \infty} -\frac{1}{N} \, \log P(N u \leq \Vfim \leq N v) = \lpv[u]. \]
\item[\parte{b}] If $\xf \leq u < v < \vf$, then
\[ \lim_{N \to \infty} -\frac{1}{N} \, \log P(N u \leq \Vfim \leq N v) = \lpv. \]
\item[\parte{c}] If $\vf < u \leq 1$, then
\[ \lim_{N \to \infty} -\frac{1}{N} \, \log P(\Vfim \geq N u) = \lpv[u]. \]
\end{itemize}
\end{teo}

\subsection{Proof of Theorem~\ref{T: PAV}}

We first observe that the continuous-time Markov chain $\{(V(t), Y(t))\}_{t \geq 0}$ proceeds according to the following transition scheme:
\begin{alignat*}{2}
&(V(t), Y(t)) \to (V(t) + 1, Y(t) + 1) 
&&\quad\text{at rate } \, Y(t) \left(N - V(t)\right), \\[0.1cm]
&(V(t), Y(t)) \to (V(t), Y(t) - 1) 
&&\quad\text{at rate } \, Y(t) \, V(t).
\end{alignat*}
The distribution of $\Vfim$ depends on $(V(t), Y(t))$ only through the embedded Markov chain, whose state space is the set of points $(r, s)$ of the two-dimensional integer lattice for which $0 \leq r \leq N$ and $0 \leq s \leq r + 1$.
The evolution of the embedded chain can be viewed as the motion of a particle through these lattice points, starting at the point $(0, 1)$ of the $xy$~plane.
(See Figure~\ref{Fig: CM MT}.)
From a point $(r, s)$ with $s > 0$, the particle moves either one step vertically downwards to $(r, s-1)$ or one step diagonally northeast (i.e., up and right) to $(r+1, s+1)$.
The probabilities of these two types of transitions are $r/N$ and $(1 - r/N)$, respectively.
Following \citet{DG}, we say that this process is \textit{strictly evolutionary}: once a state is visited and left, it is never visited again, so each state is entered either once or not at all.
The states $\{ (r, 0): r = 1, \dots, N \}$ are absorbing, thus the particle halts once it hits the $x$-axis.
If the particle hits the line $r = N$, then only the vertically downward transitions are allowed, until it reaches the point $(N, 0)$.
For each $j = 1, \dots, N$, the event $\{ \Vfim = j \}$ means that the particle hits the $x$-axis precisely at the point $(j, 0)$.

\begin{figure}[!htbp]
\centering
\includegraphics[scale=0.9]{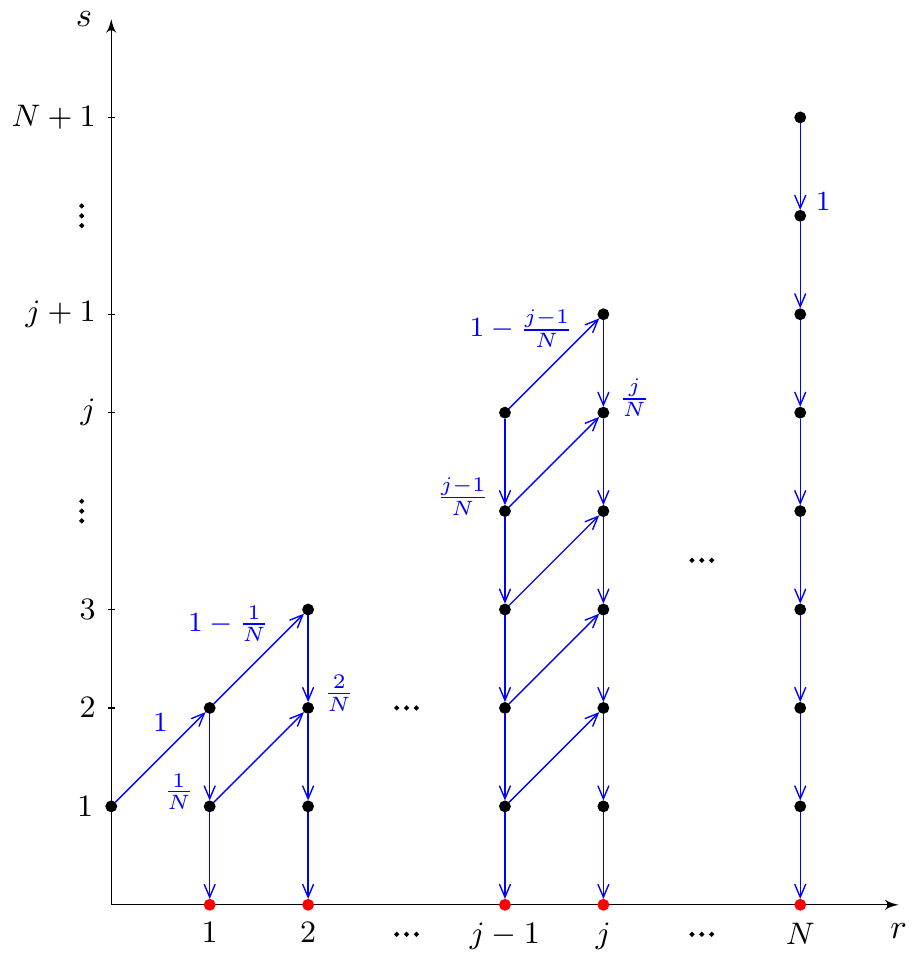}
\caption{Transitions of the embedded Markov chain.}
\label{Fig: CM MT}
\end{figure}

Now we fix $j \in \{ 1, \dots, N \}$, and define $\ccam$ to be the set of all lattice paths running from $(0, 1)$ to $(j, 0)$ that use the steps in $\{ (0, -1), (1, 1) \}$ and that hit the $x$-axis for the first time at the point $(j, 0)$.
Given any path $\cam \in \ccam$, we note that, for each $\ell \in \{ 0, \ldots, j - 1 \}$, there is exactly one northeast transition in $\cam$, which corresponds to the change of the value of the first coordinate from $\ell$ to $\ell + 1$.
Therefore, $\cam$ has $j$ northeast transitions.
As each of these transitions results in an increase by~$1$ of the second coordinate, we conclude that the number of downward transitions in $\cam$ equals~$j + 1$.
For each $i = 1, \dots, j$, let $m_i = m_i(\cam)$ denote the number of downward transitions of the path $\cam$ that are made when the first coordinate equals~$i$.
Then,
\[ P(\Vfim = j) = \sum_{\cam \in \ccam} \prod_{\ell = 0}^{j - 1} \left(\frac{N - \ell}{N}\right)
\prod_{i = 1}^{j} \left(\frac{i}{N}\right)^{m_i}. \]
Defining $a_j(\cam) = \prod_{i = 1}^{j} i^{m_i}$ (which does not depend on $N$) and $b_j = \sum_{\cam \in \ccam} a_j(\cam)$, we obtain that
\[ P(\Vfim = j) = \frac{(N - 1)!}{(N - j)!} \frac{b_j}{N^{2 j}}. \]
Clearly $b_1 = 1$.
Using that $\sum_{j = 1}^{N} P(\Vfim = j) = 1$, we see that $\{b_j\}$ satisfies the same recursive formula as $\{d_j\}$, whence the result follows. \qed

\subsection{Auxiliary results}

We present here some definitions and theorems that will be used in the proofs of Theorems~\ref{TA: CAV} and \ref{TA: GDV}.
We start off with Stirling's asymptotic estimate and bounds for factorials.

\begin{lem}
\label{L: Stirling}
\begin{itemize}
\item[\parte{a}] $n! \sim \sqrt{2\pi} \, n^{n + 1/2} \, e^{-n}$ as $n \to \infty$.
\item[\parte{b}] $\sqrt{2\pi} \, n^{n + 1/2} \, e^{-n} \leq n! \leq \sqrt{2\pi} \, n^{n + 1/2} \, e^{-n} \, e^{1/12}$ for all $n \geq 1$.
\end{itemize}
\end{lem}

\noindent
The proof of this result can be found in~\citet[Section~II.9]{Feller}; formula in part~\parte{b} is an immediate consequence of Equation~\parte{9.15} therein.

In the sequel, we will show that there is a strong relation between the numbers $d_n$ and the Stirling numbers of the second kind.
Recall that, for $m$, $n$ nonnegative integers, the \textit{Stirling number of the second kind}, denoted by $\stir{n}{m}$, is the number of ways of partitioning a set of $n$ elements into $m$ nonempty subsets.
We refer to~\citet[Section 6.1]{GKP} for more details on this subject.
The following result provides the asymptotic estimate of $\stir{2 n}{n}$ for large values of~$n$, which can be obtained by means of general techniques of analytic combinatorics (see~\citparen{Good}).

\begin{lem}
\label{L: NS}
Define the constants
\[ \al = \sqrt{\frac{1}{2 \pi \, (2 \, \vf - 1)}} \quad \text{and} \quad
\be = \frac{1}{e \, \vf \, (1 - \vf)}. \]
Then, as $n \to \infty$,
\[ \stir{2 n}{n} = \al \, \be^n \, n^{n - 1/2} \left(1 + \ogr{\frac{1}{n}}\right). \]
\end{lem}

\noindent
We observe that this lemma is a restatement of Lemma~6 in~\citet{BN} with $k = 2$ and constants $\al$ and $\be$ expressed in terms of~$\vf$.

Next, we present an asymptotic approximation and an upper bound for~$d_n$, showing how this number is related to the Stirling number $\stir{2 n}{n}$.

\begin{lem}
\label{L: Aut}
\begin{itemize}
\item[\parte{a}] As $n \to \infty$, one has
\[ d_n \sim \Ek \, n \, \stir{2 n}{n} \quad \text{where} \quad
\Ek = 2 - \frac{1}{\vf}. \]
\item[\parte{b}] $d_n \leq n \, \displaystyle\stir{2 n}{n}$ for all $n \geq 1$.
\end{itemize}
\end{lem}

\begin{proof}
This basically follows from results stated in~\citet{BN}.
This paper is devoted to the problem of the enumeration of the set $\aut$ of nonisomorphic initially connected, complete and deterministic automata of size $n$ on a $k$-letters alphabet.
Recall that the sequence $\{ d_n \}_{n \geq 1}$ refers to a similar class of automata, but with $n$ unlabelled states on an alphabet of size $k = 2$.
Thus, denoting by $\tstru$ the set of all the accessible complete and deterministic transition structures of base automata with $n$ states, we conclude from Lemma~1 in~\citep{BN} that
\[ |\aut| = 2^n \, |\tstru| = 2^n \, d_n. \]
Hence, part~\parte{a} is a consequence of the fact that
\[ |\aut| \sim \Ek \, n \, 2^n \, \stir{2 n}{n}. \]
This asymptotic estimate for $|\aut|$ was originally established by~\citet{Korshunov}, and reformulated in terms of the Stirling numbers in~\citep[Theorem~18]{BN}.
The simplified expression for the constant $\Ek$ (which is denoted in~\citep{BN} by $E_2$) is due to \citet{AEAA}.

To finish the proof, we note that from Theorem~2, Equation~\parte{1} and Lemma~9 in \citep{BN}, it follows that, for every $n \geq 1$,
\[ d_n = |\tstru| \leq n \, \displaystyle\stir{2 n}{n}, \]
which yields part~\parte{b}.
\end{proof}

\subsection{Proof of Theorem~\ref{TA: CAV}}

Now we introduce some definitions that will be used in the sequel.
We define the constants
\[ A = \frac{\sqrt{2 \, \vf - 1}}{\vf} \approx 0.9669 \quad \text{and} \quad
B = e^{-\rf} = \frac{1}{e^2 \, \vf \, (1 - \vf)} \approx 0.8359. \]
We observe that
\begin{equation}
\label{F: Rel}
A = \sqrt{2 \pi} \, \al \, \Ek \quad \text{and} \quad B = \be / e.
\end{equation}
For $v \in (0, 1)$, we define
\begin{equation*}
f(v) = \sqrt{\frac{v}{1-v}} \quad \text{and} \quad
g(v) = \exp\{ -\lpv \} = \frac{B^{v} \, v^v}{(1-v)^{1-v}}.
\end{equation*}

Theorem~\ref{TA: CAV} is an immediate consequence of the following result.

\begin{lem}
\label{L: Asymp}
\begin{itemize}
\item[\parte{a}] As $N \to \infty$, we have that
\[ P(\Vfim = N) \sim A \, B^N. \]
\item[\parte{b}] For $v \in (0, 1)$, let $\je = \je(v)$ be either $\chnv$ or $\tenv$.
Then, as $N \to \infty$,
\[ P(\Vfim = \je) \sim \frac{A}{\sqrt{2 \pi N}} \, f(\je/N) \, \left[g(\je/N)\right]^{N}. \]
\end{itemize}
\end{lem}

\begin{proof}[Proof of Lemma~\ref{L: Asymp}]
Using Theorem~\ref{T: PAV}, Lemmas~\ref{L: Stirling}, \ref{L: NS} and \ref{L: Aut}, and Equation~\eqref{F: Rel}, we obtain that, as $N \to \infty$,
\begin{align*}
P(\Vfim = N) &= \frac{N!}{N^{2 N + 1}} \, d_N \sim \frac{\sqrt{2\pi} \, N^{N + 1/2} \, e^{-N}}{N^{2 N + 1}} \, \Ek \, N \, \stir{2 N}{N} \\[0.1cm]
&\sim \frac{\sqrt{2\pi} \, \Ek \, e^{-N}}{N^{N - 1/2}} \, \al \, \be^N \, N^{N - 1/2} = A \, B^N.
\end{align*}
This shows part~\parte{a}.
Part~\parte{b} can be proved similarly (note that, since $v \in (0, 1)$, both $\je$ and $N - \je$ tend to infinity as $N \to \infty$).
\end{proof}

\subsection{Proof of Theorem~\ref{TA: GDV}}

To prove Theorem~\ref{TA: GDV}, we establish suitable asymptotic upper bounds for the absorption probabilities.
We define
\[ \fn =
\left\{
\begin{array}{cl}
f(j/N) \, \left[g(j/N)\right]^{N} & \text{for } 1 \leq j \leq N - 1, \\[0.1cm]
B^N															  & \text{for } j = N.
\end{array}	\right. \]
We will use the following results.

\begin{lem}
\label{L: LSA}
For every $u \in (0, 1)$, there exist positive constants $C_0$ and $C_1$ such that, for all sufficiently large $N$,
\[ P(\Vfim = j) \leq C_0 \, \fn \left( 1 + \frac{C_1}{N u} \right) 
\ \text{ for each } \ j = \tenu, \dots, N. \]
\end{lem}

\begin{lem}
\label{L: Est}
\begin{itemize}
\item[\parte{a}] If $u \in (0, \xf)$, then for all sufficiently large $N$,
\[ \fn \leq \fn[\tenu] 
\ \text{ for every } \ j = \tenu, \dots, \chnxf. \]
\item[\parte{b}] If $v \in (\xf, \vf)$, then for all sufficiently large $N$,
\[ \fn \leq \fn[\chnv] 
\ \text{ for every } \ j = \tenxf, \dots, \chnv. \]
\item[\parte{c}] If $u \in (\vf, 1)$, then for all sufficiently large $N$,
\[ \fn \leq \fn[\tenu] 
\ \text{ for every } \ j = \tenu, \dots, N. \]
\end{itemize}
\end{lem}

\begin{proof}[Proof of Theorem~\ref{TA: GDV}]
To prove the assertion in part~\parte{a}, we consider $0 < u < v \leq \xf$.
Since $P(N u \leq \Vfim \leq N v) \geq P(\Vfim = \tenu)$, Theorem~\ref{TA: CAV} implies
\[ \limsup_{N \to \infty} -\frac{1}{N} \, \log P(N u \leq \Vfim \leq N v) \leq \lpv[u]. \]
On the other hand, from Lemmas~\ref{L: LSA} and~\ref{L: Est}, it follows that for all sufficiently large $N$,
\[ P(N u \leq \Vfim \leq N v) = \sum_{j = \tenu}^{\chnv} P(\Vfim = j) \leq C_0 \left( 1 + \frac{C_1}{N u} \right) N \, \fn[\tenu]. \]
Consequently,
\[ \liminf_{N \to \infty} -\frac{1}{N} \, \log P(N u \leq \Vfim \leq N v) \geq \lpv[u], \]
which yields the result in part~\parte{a}.
The remaining statements are proved in a similar way.
\end{proof}

\begin{proof}[Proof of Lemma~\ref{L: LSA}]
By Lemma~\ref{L: NS}, there exists a positive constant $K$ such that, for all sufficiently large $n$,
\[ \stir{2 n}{n} \leq \al \, \be^n \, n^{n - 1/2} \left(1 + \frac{K}{n}\right). \]
Therefore, using Theorem~\ref{T: PAV} and Lemmas~\ref{L: Stirling} and \ref{L: Aut}, we conclude that, for all sufficiently large $N$,
\begin{align*}
P(\Vfim = N) &= \frac{N!}{N^{2 N + 1}} \, d_N \leq \frac{\sqrt{2\pi} \, N^{N + 1/2} \, e^{-N} \, e^{1/12}}{N^{2 N + 1}} \, N \, \stir{2 N}{N} \\[0.1cm]
&\leq \frac{\sqrt{2\pi} \, e^{1/12} \, e^{-N}}{N^{N - 1/2}} \, \al \, \be^N \, N^{N - 1/2} \left(1 + \frac{K}{N}\right) \\[0.1cm]
&= \al \, \sqrt{2\pi} \, e^{1/12} \, B^N \left(1 + \frac{K}{N}\right).
\end{align*}
Analogously, if $u \in (0, 1)$, then for all sufficiently large $N$ and each $j = \tenu, \dots, N - 1$,
\begin{align*}
P(\Vfim = j) &= \frac{N!}{(N - j)!} \frac{d_j}{N^{2 j + 1}} \leq 
\frac{N^{N + 1/2} \, e^{-N} \, e^{1/12}}{(N - j)^{N - j + 1/2} \, e^{-(N - j)}} \, \frac{1}{N^{2 j + 1}} \, j \, \stir{2 j}{j} \\[0.1cm]
&\leq \frac{e^{1/12}}{N^{1/2}} \, \frac{j \, N^{N}}{e^{j} \, N^{2 j} \, (N - j)^{N - j + 1/2}} \; \al \, \be^j \, j^{j - 1/2} \left(1 + \frac{K}{N u}\right) \\[0.1cm]
&= \frac{\al \, e^{1/12}}{N^{1/2}} \, \frac{j^{1/2}}{(N - j)^{1/2}} \, \frac{B^j \, j^j \, N^{N}}{N^{2 j} \, (N - j)^{N - j}} \left(1 + \frac{K}{N u}\right) \\[0.1cm]
&= \frac{\al \, e^{1/12}}{N^{1/2}} \, \fn \left(1 + \frac{K}{N u}\right),
\end{align*}
whence the result follows.
\end{proof}

\begin{proof}[Proof of Lemma~\ref{L: Est}]
\parte{a} We consider $N \geq 10$ (so that $\chnxf > 1$ and $\tenvf < N - 1$).
To study the properties of the function $\fn$, we can treat~$j$ as a real number in the interval $[0, N - 1]$, so that the derivatives of $\phi^{(N)}$ with respect to~$j$ can be computed.
We note that, for $j \in (0, N - 1)$,
\[ \frac{\partial \, \fn}{\partial j} = \fn \, \dn, \]
where
\[ \dn = \frac{N}{2 \, j \, (N-j)} + 
\log \left[ \frac{j/N \, (1 - j/N)}{\vf \, (1 - \vf)} \right]. \]
Also, let $j_1 < j_2 < j_3$ be given by
\[ j_1 = \frac{N}{2} \left( 1 - \sqrt{1 - {2}/{N}} \right), \quad
j_2 = \frac{N}{2}, \quad
j_3 = \frac{N}{2} \left( 1 + \sqrt{1 - {2}/{N}} \right). \]
Then, for $j \in (0, N - 1)$,
\[ \xn = \frac{\partial \, \dn}{\partial j} = \frac{(N - 2 \, j) [2 \, j \, (N - j) - N]}{2 \, j^2 \, (N - j)^2}
= \frac{2 \, (j - j_1) (j - j_2) (j - j_3)}{j^2 \, (N - j)^2}. \]
Since $j_1 < 1$, we have that $\xn > 0$ for every $j \in (1, j_2)$.
In addition, for all sufficiently large $N$,
\begin{align*}
\dn[1] &= \frac{N}{2 \, (N - 1)} + 
\log \left[ \frac{1/N \, (1 - 1/N)}{\vf \, (1 - \vf)} \right] < 0, \quad \text{and} \\[0.1cm]
\dn[j_2] &= \frac{2}{N} + 
\log \left[ \frac{1}{4 \, \vf \, (1 - \vf)} \right] > 0.
\end{align*}
Now fix $u \in (0, \xf)$.
The assertion in part~\parte{a} is proved once we show that for all sufficiently large $N$,
\begin{equation}
\label{F: Cond-a}
\fn[\chnxf] \leq \fn[\tenu].
\end{equation}
To prove this, we define for $z \in (0, 1)$,
\begin{equation*}
\fmn = [f(z)]^{1/N} \, g(z) = \left(\frac{z}{1-z}\right)^{1/(2 N)} \, \exp\{ -\lpv[z] \}.
\end{equation*}
By Dini's Theorem, as $N \to \infty$, the sequence $\{ \Gamma^{(N)} \}$ converges to $g$ uniformly on each closed interval contained in the interval $(0, 1/2)$.
Consequently,
\[ \lim_{N \to \infty} [\fn[\chnxf]]^{1/N} = g(\xf) < g(u) = \lim_{N \to \infty} [\fn[\tenu]]^{1/N}. \]
This implies that~\eqref{F: Cond-a} holds true for all sufficiently large~$N$.

\bigskip
\noindent
\parte{b} Fixed $v \in (\xf, \vf)$, the statement in \parte{b} follows from the facts that the function~$\phi^{(N)}$ is continuous on the interval $[\tenxf, \chnv]$, and that
\[ \dn > \log \left[ \frac{j/N \, (1 - j/N)}{\vf \, (1 - \vf)} \right] > 0 \]
for every $j \in (\tenxf, \chnv)$.

\bigskip
\noindent
\parte{c} We first observe that for every $N \geq 2$,
\[ B < 1 \leq \frac{(N - 1)^{N - 1/2}}{N^{N - 2}}, \]
which implies that $\fn[N] \leq \fn[N - 1]$.
Now fixed $u \in (\vf, 1)$, we prove that for all sufficiently large $N$, the function $\phi^{(N)}$ is decreasing on the interval $[\tenu, N - 1]$.
As $j_3 > N - 1$, we conclude that $\xn < 0$ for every $j \in (j_2, N - 1)$.
Thus, it is enough to show that for all sufficiently large $N$,
\begin{equation}
\label{F: Cond-c}
\dn[\tenu] < 0.
\end{equation}
To accomplish this, we define for $z \in (0, 1)$,
\begin{align*}
\lan &= \frac{1}{2 N z \, (1 - z)} + 
\log \left[ \frac{z \, (1 - z)}{\vf \, (1 - \vf)} \right], \quad \text{and} \\[0.1cm]
\lal &= \log \left[ \frac{z \, (1 - z)}{\vf \, (1 - \vf)} \right].
\end{align*}
By Dini's Theorem, the sequence $\{ \Lambda^{(N)} \}$ converges to $\Lambda$ as $N \to \infty$, uniformly on each closed interval contained in the interval $(0, 1)$.
Hence,
\[ \lim_{N \to \infty} \dn[\tenu] = \lim_{N \to \infty} \Lambda^{(N)}\left(\frac{\tenu}{N}\right) = \lal[u] < 0. \]
From this, it follows that~\eqref{F: Cond-c} is valid for all sufficiently large $N$.
\end{proof}

\subsection{Proof of Theorem~\ref{T: PGD}}
\label{SS: Proof PGD}

For a set $K \subset [0, \infty)$, we denote by $\ic$ the infimum of $\ldx$ over $K$.
To prove part~\parte{a}, let $F$ be a nonempty closed subset of $[0, \infty)$.
If $\ic[F] = 0$ or $\ic[F] = \infty$, there is nothing to prove.
Assume that $0 < \ic[F] < \infty$, and define
\begin{equation}
\label{F: Pontos}
x_{1} = \sup \, (F \cap [0, \xf]), \:
x_{2} = \inf \, (F \cap [\xf, \vf]), \:
x_{3} = \sup \, (F \cap [\vf, \infty)).
\end{equation}
By the monotonicity properties of $H$, we have that $\ic[F] = \ldx[x_1] \, \wedge \, \ldx[x_2] \, \wedge \, \ldx[x_3]$ (we suppose that none of the intersections in~\eqref{F: Pontos} is empty; otherwise, the corresponding term is missing).
Using Theorem~\ref{TA: GDX}, we get
\begin{align*}
&\limsup_{N \to \infty} \frac{1}{N} \, \log \med(F) \leq 
\limsup_{N \to \infty} \frac{1}{N} \, \log \left( \med([0, x_1]) + \med([x_2, \vf]) + \med([\vf, x_3]) \right) \\[0.1cm]
&= \max \left\{ \limsup_{N \to \infty} \frac{1}{N} \, \log \med([0, x_1]),
\limsup_{N \to \infty} \frac{1}{N} \, \log \med([x_2, \vf]),
\limsup_{N \to \infty} \frac{1}{N} \, \log \med([\vf, x_3]) \right\} \\[0.2cm]
&= \max \left\{ -\ldx[x_1], -\ldx[x_2], -\ldx[x_3] \right\} = -\ic[F].
\end{align*}
This establishes part~\parte{a}.

Regarding part~\parte{b}, let $G$ be a nonempty open subset of $[0, \infty)$.
We will show that for each $x \in G$,
\begin{equation}
\label{F: LI}
\liminf_{N \to \infty} \frac{1}{N} \, \log \med(G) \geq -\ldx.
\end{equation}
As~\eqref{F: LI} trivially holds for $x \in \Gui$, it is enough to prove it for $x \in \Gzu$.

First, assume that $x \in \Gzx$.
Then, there exists $y < x$ such that $(y, x] \subset \Gzx$.
Using that $P(\Xfim \leq N x) = P(\Xfim \leq N y) + \med((y, x])$ and Theorem~\ref{TA: GDX}, we obtain
\[ -\ldx \leq \max \left\{ -\ldx[y], \, \liminf_{N \to \infty} \frac{1}{N} \, \log \med((y, x]) \right\}. \]
Since $\ldx < \ldx[y]$, it follows that~\eqref{F: LI} holds true.

Now if $x \in \Gxv$, then there exists $y < x$ such that $[y, x] \subset \Gxv$.
Consequently,
\[ \liminf_{N \to \infty} \frac{1}{N} \, \log \med(G) \geq
\liminf_{N \to \infty} \frac{1}{N} \, \log \med([y, x]) = -\ldx[y] \geq -\ldx. \]
A similar argument shows that~\eqref{F: LI} also holds for each $x \in \Gvu$.
Since $G$ is open and $H$ is continuous on $\xf$, we conclude that~\eqref{F: LI} is valid for every $x \in \Gzu$.
This completes the proof of part (b). \qed

\end{document}